\documentclass[11pt]{article}
\usepackage{amsmath,amssymb,amscd,color}
\newcommand{\bq}{\begin{equation}}
\newcommand{\eq}{\end{equation}}

\newcommand{\R}{{ \mathbb{R}  }}

\begin{document}
\bibliographystyle{plain}

\newtheorem{defn}{Definition}
\newtheorem{lemma}[defn]{Lemma}
\newtheorem{proposition}{Proposition}
\newtheorem{theorem}[defn]{Theorem}
\newtheorem{cor}{Corollary}
\newtheorem{remark}{Remark}
\numberwithin{equation}{section}

\def\Xint#1{\mathchoice
   {\XXint\displaystyle\textstyle{#1}}%
   {\XXint\textstyle\scriptstyle{#1}}%
   {\XXint\scriptstyle\scriptscriptstyle{#1}}%
   {\XXint\scriptscriptstyle\scriptscriptstyle{#1}}%
   \!\int}
\def\XXint#1#2#3{{\setbox0=\hbox{$#1{#2#3}{\int}$}
     \vcenter{\hbox{$#2#3$}}\kern-.5\wd0}}
\def\ddashint{\Xint=}
\def\dashint{\Xint-}
\def\aint{\Xint\diagup}

\newenvironment{proof}{{\bf Proof.}}{\hfill\fbox{}\par\vspace{.2cm}}
\newenvironment{pfthm1}{{\par\noindent\bf
            Proof of Theorem 1. }}{\hfill\fbox{}\par\vspace{.2cm}}
\newenvironment{pfthm2}{{\par\noindent\bf
            Proof of Theorem 2. }}{\hfill\fbox{}\par\vspace{.2cm}}
\newenvironment{pfthm3}{{\par\noindent\bf
Proof of Theorem 3. }}{\hfill\fbox{}\par\vspace{.2cm}}
\newenvironment{pfthm4}{{\par\noindent\bf
Proof of Theorem \ref{theorem3}. }}{\hfill\fbox{}\par\vspace{.2cm}}
\newenvironment{pfthm5}{{\par\noindent\bf
Proof of Theorem 5. }}{\hfill\fbox{}\par\vspace{.2cm}}
\newenvironment{pflemsregular}{{\par\noindent\bf
            Proof of Lemma \ref{sregular}. }}{\hfill\fbox{}\par\vspace{.2cm}}

\title{On the blow-up criterion and small data global existence for the Hall-magnetohydrodynamics }
\author{$\mbox{Dongho Chae}^{(*)}\mbox{ and Jihoon Lee}^{(\dagger)}$\\
 Department of Mathematics;\\
 Chung-Ang University\\
 Seoul 156-756; Korea\\
 e-mail : ($*$)dchae@cau.ac.kr,\quad ($\dagger$)jhleepde@cau.ac.kr}

\date{}

\maketitle
\begin{abstract}
  In this paper, we establish an optimal blow-up criterion for classical solutions to the incompressible resistive Hall-magnetohydrodynamic equations. We also prove two  global-in-time existence results of the classical solutions for small initial data, the smallness conditions of which are given by the suitable Sobolev and the Besov norms respectively. Although the Sobolev space version is already an improvement of the corrresponding result in \cite{Chae-Degond-Liu}, the optimality in terms of the scaling property is achieved via the Besov space estimate.
  The special property of the energy estimate in terms of $\dot{B}^s_{2,1}$ norm is essential for this result.
  Contrary to the usual MHD  the global well-posedness in the $2\frac12$ dimensional Hall-MHD is wide open.
\end{abstract}

\section{Introduction}
 \setcounter{equation}{0}

\indent
 We study  the following three dimensional incompressible resistive viscous Hall-magnetohydrodynamics system(Hall-MHD):
\begin{eqnarray}
&& \partial_t u +(u\cdot \nabla) u +\nabla p=(\nabla \times B)\times B+\Delta u, \label{eq1}\\
&& \nabla \cdot u=0, \qquad\qquad\qquad\qquad\qquad\qquad\qquad\qquad\qquad\qquad \mbox{ in } \R^3 \times (0, \infty),\label{eq2}\\
&& \partial_t B -\nabla \times (u \times B)+\nabla \times ((\nabla \times B)\times B)=\Delta B, \label{eq3}\\
&& (u(0, x),\, B(0, x))=(u_0(x),\, B_0(x)), \qquad\qquad\qquad\qquad \quad\mbox{ in }\R^3,\label{eq4}
\end{eqnarray}
 where $u $, $B$ and $p$ represent $3$-dimensional velocity vector field, the magnetic field and scalar pressure, respectively. The initial data $u_0$ and $B_0$ satisfy
\[ \nabla \cdot u_0 =\nabla \cdot B_0=0.
\]
Note that if $\nabla \cdot B_0=0$, then the divergence free condition is propagated by (\ref{eq3}). Comparing with the usual MHD equations, the Hall-MHD equations have the Hall term $\nabla \times((\nabla \times B)\times B)$ in \eqref{eq3}, which plays an important role in magnetic reconnection which is happening in the case of large magnetic shear. The Hall magnetohydrodynamics was studied systematically by Lighthill\cite{Light}. In particular, he considered Alfv\'{e}n waves with Hall effect. The Hall-MHD is important describing many physical phenomena, e.g., space plasmas, star formation, neutron stars and geo-dynamo (see \cite{BT, For, HG, MGM, SU, War} and references therein).
 \\
In \cite{ADFL}, Acheritogaray, Degond, Frouvelle and Liu derived the Hall-MHD equations from either two fluids model or kinetic models in a mathematically rigorous way. In \cite{Chae-Degond-Liu}, the global existence of weak solutions to \eqref{eq1}--\eqref{eq4} as well as the local well-posedness of classical solution are established. Also, a blow-up criterion for smooth solution to \eqref{eq1}--\eqref{eq4} and the global existence of smooth solution for small initial data are obtained (see \cite[Theorem 2.2 and 2.3]{Chae-Degond-Liu}). Very recently, temporal decay for the weak solution and smooth solution with small data to Hall-MHD are also established in \cite{Chae-Schon}.\\
Our goal of this paper is to improve in an optimal way the blow-up criterion and global existence of smooth solution with small initial data to the Hall-MHD equations \eqref{eq1}--\eqref{eq4} derived in \cite{Chae-Degond-Liu}. The sense of optimality is explained in detail in Remark \ref{rem1}.\\
Using vector identity, we can rewrite \eqref{eq1}--\eqref{eq4} as follows:
\begin{eqnarray}
&& \partial_t u +(u\cdot \nabla) u +\nabla \left(p+\frac{|B|^2}{2} \right)=(B\cdot \nabla)B+\Delta u, \label{eq1-1}\\
&& \nabla \cdot u=0, \qquad\qquad\qquad\qquad\qquad\qquad\qquad\qquad\qquad\qquad \mbox{ in } \R^3\times (0, \infty),\label{eq2-1}\\
&& \partial_t B +(u\cdot\nabla )B+\nabla \times ((\nabla \times B)\times B)=(B\cdot \nabla)u+\Delta B, \label{eq3-1}\\
&& (u(0, x),\, B(0, x))=(u_0(x),\, B_0(x)), \qquad\qquad\qquad\qquad \quad\mbox{ in }\R^3.\label{eq4-1}
\end{eqnarray}
Note that a weak solution $(u, B)$ to \eqref{eq1}--\eqref{eq4} satisfies the following energy inequality (see \cite{Chae-Degond-Liu}):
\begin{equation}\label{ineq-energy}
\frac12(\| u(t) \|_{L^2}^2 +\| B(t)\|_{L^2}^2) +\int_0^t \| \nabla u(\cdot, \tau)\|_{L^2}^2 +\| \nabla B(\cdot, \tau)\|_{L^2}^2 d \tau \leq \frac12(\| u_0\|_{L^2}^2+\| B_0 \|_{L^2}^2)
\end{equation}
for almost every $t \in [0, \infty)$.\\
Our first result is Serrin type\cite{Ser} criterion for the solutions to \eqref{eq1}--\eqref{eq4}.
\begin{theorem}\label{thm1}
Let $m >\frac52$ be an integer and $( u_0,\, B_0) \in H^{m}({\mathbb{R}}^3)$ with $\nabla \cdot u_0=\nabla\cdot B_0=0$. Then for the first blow-up time $T^{*}<\infty$ of the calssical solution to \eqref{eq1}--\eqref{eq4}, it holds that  \[
\limsup_{t \nearrow T^{*}} (\| u(t)\|_{H^m}^2+\| B(t)\|_{H^m}^2)=\infty,
\]
if and only if
\[
\| u\|_{L^q(0, T^{*}; L^p(\R^3))}+\| \nabla B \|_{L^{\gamma}(0, T^{*}; L^{\beta}(\R^3))}=\infty,
\]
where $p$, $q$, $\beta$ and $\gamma$ satisfy the relation
\[
\frac{3}{p}+\frac{2}{q} \leq 1, \,\, \frac{3}{\beta}+\frac{2}{\gamma}\leq 1,\mbox{ and } p,\, \beta \in (3, \infty].
\]
\end{theorem}
Next, we consider an improvement of Theorem \ref{thm1} for the case $q=\gamma=2$ and $p=\beta=\infty$ by using BMO space, in which $L^{\infty}(\R^d)$ is embedded (see \cite{Trieb} for the definition and properties).
\begin{theorem}\label{thm2}
Let $m >\frac52$ be an integer and $( u_0,\, B_0) \in H^{m}({\mathbb{R}}^3)$ with $\nabla \cdot u_0=\nabla \cdot B_0=0$. Then for the first blow-up time $T^{*}<\infty$ of the classical solution to \eqref{eq1}--\eqref{eq4}, it holds that  \[
\limsup_{t \nearrow T^{*}} (\| u(t)\|_{H^m}^2+\| B(t)\|_{H^m}^2)=\infty,
\]
if and only if
\[
\int_0^{T^{*}} (\| u\|_{BMO}^2+\| \nabla B\|_{BMO}^2)dt=\infty.
\]
\end{theorem}
We consider the homogeneous Sobolev space $\dot{H}^s(\R^d)$ ($s \in \R$), which is defined as follows:\\
For any tempered distribution $f$ on $\R^d$, we define the seminorm
\[
\| f\|_{\dot{H}^s}=\| \Lambda^s f\|_{L^2} =\left( \int_{\R^d} |\xi|^{2s} |\hat{f}(\xi)|^2 d\xi \right)^{\frac12},
\]
where $\Lambda$ denotes $(-\Delta)^{\frac12}$
and $\dot{H}^s(\R^d)$ is the set of all functions $f$, for which $\| f\|_{\dot{H}^s}$ is finite.
Using the above notation, we state the global in time existence of smooth solution with small data.
\begin{theorem}\label{thm3}
Let $m > \frac52$ be an integer and $(u_0, B_0) \in H^m(\R^3)$ with $\nabla\cdot u_0=\nabla \cdot B_0=0$. There exists a constant $K>0$ such that if $\|  u_0 \|_{\dot{H}^{\frac32}}+\|  B_0 \|_{\dot{H}^{\frac32}} <K$, then there exists a unique global classical solution $(u,\, B) \in L^{\infty}(0, \infty; H^m(\R^3))$ to \eqref{eq1}--\eqref{eq4}.
\end{theorem}
\begin{remark}\label{rem1}
Because Hall-term breaks the natural scaling of the Navier-Stokes equations, there does not exist the scaling invariant function spaces for the Hall-MHD equations \eqref{eq1}--\eqref{eq4}. If we set the fluid velocity  $u\equiv0$, then \eqref{eq3} is reduced to
\begin{equation}\label{eq-magneto}
\partial_t B+\nabla\times ((\nabla \times B)\times B)=\Delta B,\qquad\qquad \mbox{ in } \R^3\times (0, \, \infty),
\end{equation}
with $B(0, x)=B_0(x)$.\\ If we consider magnetic field $B$ satisfying \eqref{eq-magneto}, then the function $B_{\lambda}(x,t):= B(\lambda x, \lambda^2 t)$ form a solution to \eqref{eq-magneto} again. Therefore, $\nabla B$ of \eqref{eq-magneto} has the same scaling with the fluid velocity $u$ to the usual Navier-Stokes equations and $\| \nabla B \|_{L^{\gamma}(0, T^{*}; L^{\beta}(\R^3))}$ with $\frac{3}{\beta}+\frac{2}{\gamma}= 1,\mbox{ and } \beta \in (3, \infty]$ in Theorem \ref{thm1} is scaling invariant with respect to the above scaling. If we consider scaling of 3D Navier-Stokes equations and \eqref{eq-magneto}, then it seems natural to expect global existence of a classical solution with the small data $\| u_0\|_{\dot{H}^{\frac12}}+\| B_0 \|_{\dot{H}^{\frac32}}< \epsilon$. But due to the lack of the cancellation property when we deal with $u$ equation in $\dot{H}^{\frac12}$ and Hall equation in $\dot{H}^{\frac32}$, it seems difficult to
obtain the result in this space. If we use the suitable Besov spaces, however, we could overcome this difficulty, and could prove the following optimal
small data global well-posedness result as in the next theorem.
\end{remark}
\begin{theorem}\label{theorem3}
Let $m > \frac52$ be an integer and $(u_0, B_0) \in H^m(\R^3)$ with $\nabla\cdot u_0=\nabla \cdot B_0=0$. There exists a constant $\epsilon=\epsilon(\| B_0 \|_{L^2})>0$ such that if $\|  u_0 \|_{\dot{B}^{\frac12}_{2,1}} +\|  B_0 \|_{\dot{B}^{\frac32}_{2,1}}<\epsilon$, then there exists a unique global classical solution $(u,\, B) \in L^{\infty}(0, \infty; H^m(\R^3))$ to \eqref{eq1}--\eqref{eq4}.
\end{theorem}
\begin{remark}\label{rem2}
We consider $2\frac12 D$ flows as in \cite{MB} for the Hall-MHD, i.e., \[
 u: \R^2\times (0, \infty) \rightarrow \R^3\qquad B:\R^2\times (0, \infty) \rightarrow \R^3.
\]
We denote $\tilde{u}=(u_1,\, u_2)$, $\tilde{B}=(B_1,\, B_2)$, $\tilde{\nabla}=\left( \partial_1,\, \partial_2 \right)$, $\tilde{\Delta}=\partial_1^2+\partial_2^2$ and $j=\nabla\times B=\left(\begin{array}{ccc} \partial_2 B_3\\-\partial_1 B_3\\ \partial_1 B_2-\partial_2 B_1 \end{array}\right)$.\\
Using the above notation, \eqref{eq1}--\eqref{eq3} are reduced to the following system (considering the initial data satisfies $\tilde{\nabla}\cdot \tilde{B}_0=0$).
\begin{eqnarray}
&& \partial_t u +(\tilde{u}\cdot \tilde{\nabla}) u +\tilde{\nabla} \left(p+\frac{|B|^2}{2}\right)=(\tilde{B}\cdot\tilde{\nabla}) B+\tilde{\Delta} u, \label{eq1-3}\\
&& \tilde{\nabla} \cdot \tilde{u}=0, \qquad\qquad\qquad\qquad\qquad\qquad\qquad\qquad\qquad\qquad \mbox{ in }  \R^2\times (0, \infty),\label{eq2-3}\\
&& \partial_t B+ (\tilde{u}\cdot\tilde{\nabla})B+(\tilde{B} \cdot\tilde{\nabla})j-(\tilde{j}\cdot\tilde{\nabla})B=\tilde{\Delta} B+(\tilde{B}\cdot\tilde{\nabla})u, \label{eq3-3}
\end{eqnarray}
For $2\frac12 D$ Navier-Stokes flows, the global existence of a classical solution is immediate, since $\tilde{u}$ is solution to 2D Navier-Stokes equations and $u_3$ is a solution to the linear scalar diffusion equation. For $2\frac12 D$ usual MHD flows, the global existence can be obtained from the blow-up criterion, standard embedding
\[
\dot{H}^1(\R^2) \hookrightarrow BMO(\R^2),
\]
  and energy inequality
\[
\int_0^T \| u(t)\|_{BMO}^2 dt \leq C\int_0^T \| \tilde{\nabla} u (t)\|_{L^2}^2 dt< \infty,\mbox{ for any } T>0.
\]
Also for the both equations, planar component equations and $u_3$ or $B_3$ equation are decoupled. So if we have the global existence of a classical solution to the equations of the planar part, then we can solve the equation of the third component.
But for the $2\frac12 D$  Hall-MHD flows, planar part $(\tilde{u},\tilde{B})$ and the third components are related intimately through the Hall term $(\tilde{B} \cdot\tilde{\nabla})j-(\tilde{j}\cdot\tilde{\nabla})B$. Hence, it seems a challenging problem whether there exists a finite time blowup of a classical solution or not for $2\frac12 D$ Hall-MHD flows \eqref{eq1-3}--\eqref{eq3-3}. In this case, if we use Theorem \ref{thm2}, we reduce the blow-up criterion as follows: With the same assumptions in Theorem \ref{thm2}, it holds that
\[
\limsup_{t \nearrow T^{*}} (\| u(t)\|_{H^m}^2+\| B(t)\|_{H^m}^2)=\infty,
\]
if and only if
\[
\int_0^{T^{*}} \| j\|_{BMO}^2dt=\infty.
\]
\end{remark}
\section{Preliminaries}\label{sect2}
We first set our notations, and recall definitions and properties of the Besov spaces. We follow \cite{Trieb}. Let ${\mathcal{S}}$ be the Schwartz class of rapidly decreasing functions. Given $f \in {\mathcal{S}}$, its Fourier transform $\hat{f}$ is defined by
\[
\hat{f}(\xi)=\frac{1}{(2 \pi)^d}\int_{\R^d} e^{-ix \cdot \xi} f(x) dx.
\]
We consider the homogeneous Besov spaces $\dot{B}^{s}_{p,q}$ ($p,q \in [1, \infty]$), which is defined as follows. Let $\{ \psi_k \}_{ k \in {\mathbb{Z}}}$ be the Littlewood-Paley partition of unity, whose Fourier transform $\hat{\psi}_k (\xi)$ is supported on the annulus $\{ \xi \in {\mathbb{R}}^d\, |\, 2^{k-1} \leq |\xi| \leq 2^k \}$.  Then the homogeneous Besov semi-norm $\| f \|_{\dot{B}^{s}_{p,q}}$ is defined by
\[
\| f \|_{\dot{B}^{s}_{p,q}}=\left\{ \begin{array}{ll}  [\sum_{-\infty}^{\infty} 2^{kqs}\| \psi_k * f \|_{L^p}^q]^{\frac{1}{q}}\mbox{ if }q \in [1,\, \infty)\\ \sup_{k} [2^{ks} \| \psi_k * f\|_{L^p}]\mbox{ if } q=\infty.    \end{array}\right.
\]
Also the homogeneous Triebel-Lizorkin space $\dot{F}^{s}_{p,q}$ is a quasi-normed space with semi-norm $\| \cdot \|_{\dot{F}^s_{p,q}}$ which is defined by
\[
\| f \|_{\dot{F}^{s}_{p,q}}=\left\{ \begin{array}{ll}  \left\|\left(\sum_{k \in {\mathbb{Z}}} 2^{kqs}| \psi_k * f (\cdot) |^q\right)^{1/q}\right\|_{L^p}\mbox{ if }q \in [1,\, \infty)\\ \left\|\sup_{k\in {\mathbb{Z}}} (2^{ks} | \psi_k * f|)\right\|_{L^p}\mbox{ if } q=\infty.    \end{array}\right.
\]
\begin{proposition}\label{Besov}
{\rm{(i)}} Bernstein's Lemma: Assume that $f \in L^p$ with $1\leq p \leq \infty$, and $\mbox{supp }\hat{f} \subset \{ 2^{j-2} \leq |\xi| <2^j\}$, then there exists a constant $C_k$ such that the following inequalities hold:
\[
C_k^{-1} 2^{jk} \| f\|_{L^p} \leq \| D^k f\|_{L^p} \leq C_k 2^{jk} \| f\|_{L^p}.
\]
Assume that $f \in L^p$ with $1\leq p \leq q\leq \infty$, and $\mbox{supp }\hat{f} \subset \{ |\xi| <2^j\}$, then there exists a constant $C$ such that the following inequality holds:
\[
\| f\|_{L^q} \leq C2^{jd\left( \frac{1}{p}-\frac{1}{q}\right)} \| f\|_{L^p}.
\]
{\rm{(ii)}} We have the equivalence of norms
\[
\| D^k f \|_{\dot{B}^{s}_{p,q}} \sim \| f\|_{\dot{B}^{s+k}_{p,q}}.
\]
{\rm{(iii)}} Let $s>0,\, q \in [1, \infty]$, then there exists a constant $C$ such that the following inequalities holds :
\[
\| fg\|_{\dot{B}^{s}_{p,q}}\leq C\left( \| f\|_{L^{p_1}}\|g\|_{\dot{B}^{s}_{p_2,q}}+\| g\|_{L^{r_1}}\|f\|_{\dot{B}^{s}_{r_2,q}}  \right),
\]
for homogeneous Besov spaces, where $p_1,\, r_1 \in [1,\infty]$ such that $\frac{1}{p} =\frac{1}{p_1}+\frac{1}{p_2}=\frac{1}{r_1}+\frac{1}{r_2}$.\\
Let $s_1,s_2 \leq \frac{n}{p}$ ($n$ is the dimension) such that $s_1 +s_2 >0$, $f \in \dot{B}^{s_1}_{p,1}$ and $g \in \dot{B}^{s_2}_{p,1}$. Then $fg \in \dot{B}^{s_1+s_2-\frac{n}{p}}_{p,1}$ and
\[
\| fg\|_{ \dot{B}^{s_1+s_2-\frac{n}{p}}_{p,1}}\leq C\| f\|_{\dot{B}^{s_1}_{p,1}}\| g\|_{ \dot{B}^{s_2}_{p,1}}.
\]
{\rm{(iv)}} For $s \in (-\frac{n}{p}-1, \frac{n}{p}]$ we have
\[
\| [u, \Delta_q ] w \|_{L^p} \leq c_q 2^{-q(s+1)} \| u \|_{\dot{B}^{\frac{n}{p}+1}_{p,1}} \| w\|_{\dot{B}^{s}_{p,1}}
\]
with $\sum_{q \in {\mathbb{Z}}} c_q \leq 1$. In the above, we denote
\[
[u, \Delta_q] w =u\Delta_qw -\Delta_q(uw).
\]
{\rm{(v)}} We have the following interpolation inequalities for $s, s_1>0$, $s=\theta s_1$ and $\theta\in (0, 1)$:
\[
\| f \|_{\dot{B}^{s}_{2,1}} \leq C \| f \|_{\dot{B}^{s_1}_{2,1}}^{\theta} \| f\|_{L^2}^{1-\theta}.
\]
{\rm{(vi)}} Suppose $ \nabla f\in BMO(\R^d)$ and $f \in  L^2$. Then we have
\[
\| f \|_{BMO} \leq C(\| \nabla f \|_{BMO}+ \| f\|_{L^2}).
\]
\end{proposition}
\begin{proof}
The proof of (i)--(iv) is rather standard and we can find the proofs in many references (e.g., see \cite{Chae-Lee} and references therein). Although the proofs for (v) and (vi) looks standard, we could not find a literature including these, and we present  their proofs here.
In order to prove (v) we write:
\[
\| f \|_{\dot{B}^{s}_{2,1}}= \sum_{q \in {\mathbb{Z}}} 2^{qs} \| \Delta_q f \|_{L^2} =\sum_{q \leq N}2^{qs} \| \Delta_q f \|_{L^2}+\sum_{q >N} 2^{qs} \| \Delta_q f \|_{L^2}
\]
\[
\leq \left(\sum_{q\leq N}2^{2qs}\right)^{\frac12} \left(  \sum_{q\in {\mathbb{Z}}}\| \Delta_q f \|_{L^2}^2 \right)^{\frac12}+C2^{N(s-s_1)}\sum_{q >N} 2^{qs_1} \| \Delta_q f\|_{L^2}
\]
\[
\leq C2^{Ns}\| f\|_{L^2}+C2^{N(s-s_1)} \| f \|_{\dot{B}^{s_1}_{2,1}}.
\]
Let us choose $N=\left[ \log_2\left(  \frac{\| f \|_{\dot{B}^{s_1}_{2,1}}}{\| f\|_{L^2}}  \right)^{\frac{1}{s_1}}   \right]$, i.e., satisfying
\[
2^{Ns}\| f\|_{L^2}=2^{N(s-s_1)} \| f \|_{\dot{B}^{s_1}_{2,1}}.
\]
Then we have the interpolation inequality in (v). Since we know $BMO=\dot{F}^{0}_{\infty, 2}$ (see pp. 243--244 of \cite{Trieb}) and the Fourier transform of $\sum_{k<0} |\psi_k *f|^2$ is supported in the unit ball, we deduce, using (i), that
\begin{eqnarray*}
\| f \|_{BMO} & \leq & \left\| \left(  \sum_{k \geq 0} |\psi_k *f |^2 \right)^{1/2}\right\|_{L^{\infty}} + \left\| \left(  \sum_{k < 0} |\psi_k *f |^2 \right)^{1/2}\right\|_{L^{\infty}}\\
&\leq & \left\| \left(  \sum_{k \geq 0} 2^{2k}|\psi_k *f |^2 \right)^{1/2}\right\|_{L^{\infty}}+C \left\| \left( \sum_{k < 0} |\psi_k *f |^2 \right)^{1/2}\right\|_{L^2}\\
&\leq & C(\| \nabla f \|_{BMO}+ \| f\|_{L^2}).
\end{eqnarray*}
This completes the proof of (vi).
\end{proof}
\section{Regularity Criterion}
Through this paper, $C$ denotes a generic positive constant which may change from one line to the other.\\
We first derive blow-up criterion for classical solutions to \eqref{eq-magneto} as mentioned in Remark \ref{rem1}, which is called the Hall equation.
\begin{proposition}\label{prp1}
Let $m >\frac52$ be an integer and $ B_0 \in H^{m}({\mathbb{R}}^3)$. Then for the first blow-up time $T^{*}<\infty$ of the classical solution to \eqref{eq-magneto}, it holds that  \[
\limsup_{t \nearrow T^{*}} \| B(t)\|_{H^m}^2=\infty,
\]
if and only if
\[
\| \nabla B \|_{L^{\gamma}(0, T^{*}; L^{\beta}(\R^3))}=\infty,
\]
where $\beta$ and $\gamma$ satisfy the relation
\[
 \frac{3}{\beta}+\frac{2}{\gamma}\leq 1,\mbox{ and } \beta \in (3, \infty].
\]
\end{proposition}
\begin{proof} Taking $D^{\alpha}=\partial^{|\alpha|}/\partial x_1^{\alpha_1}\partial x_{2}^{\alpha_2}\partial x_3^{\alpha_3}$ operator where $\alpha=(\alpha_1, \alpha_2, \alpha_3)\in ({\mathbb{N}}\cup \{0\})^3$ with $|\alpha|=\alpha_1+\alpha_2+\alpha_3 \leq m$ with
$m \geq 3$, scalar product with $D^{\alpha} B$, and sum over $\alpha$ with $|\alpha| \leq m$, we have
\[
\frac12 \frac{d}{dt} \|  B \|_{H^m}^2 +\| \nabla B \|_{H^m}^2
= -\sum_{ |\alpha| \leq m}\int_{\R^3} D^{\alpha} ((\nabla \times B)\times B)\cdot D^{\alpha} (\nabla \times B) dx=J .\]
Using the cancellation property, we have
\[
J=-\sum_{0\leq |\alpha| \leq m}\int_{\R^3} [D^{\alpha} ((\nabla \times B)\times B)-(D^{\alpha}(\nabla \times B))\times B]\cdot D^{\alpha} (\nabla \times B) dx
\]
From the calculus inequality, interpolation inequality and Young's inequality, we have
\begin{eqnarray*}
|J| &\leq & C \| \nabla B \|_{L^{\beta}} \|  B \|_{W^{m, \frac{2\beta}{\beta-2}}} \| \nabla B \|_{H^m}\\
&\leq & C\| \nabla B \|_{L^{\beta}} \|  B \|_{H^m}^{\frac{\beta-3}{\beta}} \|  \nabla B \|_{H^m}^{\frac{\beta+3}{\beta}}\\
&\leq & C\| \nabla B \|_{L^{\beta}}^{\frac{2\beta}{\beta-3}}\|  B \|_{H^m}^2+\frac12\|  \nabla B \|_{H^m}^2.
\end{eqnarray*}
Consequently, one has
\[
\frac{d}{dt}\|  B \|_{H^m}^2 +\| \nabla B \|_{H^m}^2 \leq C\| \nabla B \|_{L^{\beta}}^{\frac{2\beta}{\beta-3}}\|  B \|_{H^m}^2.
\]
Using the Gronwall's inequality, we have\[
\sup_{ 0\leq t \leq T} \| B (t) \|_{H^m}^2 \leq \| B_0 \|_{H^m}^2 \exp \left( C\int_0^T \| \nabla B \|_{L^{\beta}}^{\frac{2\beta}{\beta-3}}dt \right).\]
Since $\frac{3}{\beta}+\frac{\beta-3}{\beta}=1$, this completes the proof.
\end{proof}
Next, we derive a priori estimates for the proof of our theorems.
\begin{proposition}\label{prp2}
Let $m >\frac52$ be an integer and $(u, B)$ be a smooth solution to \eqref{eq1-1}--\eqref{eq4-1}. Then there exist two universal constants $C_1$ and $C_2$ such that the following a priori estimates hold:
\[
\frac{d}{dt}(\|u\|_{H^1}^2 +\| B\|_{H^1}^2) + \| \nabla u \|_{H^1}^2+\| \nabla B \|_{H^1}^2
\]
\bq\label{apriori1}
\leq C_1 (\| u\|_{L^p}^{\frac{2p}{p-3}} +\| \nabla B \|_{L^{\beta}}^{\frac{2\beta}{2\beta-3}})(\| u \|_{H^1}^2+\| B \|_{H^1}^2),
\eq
and
\[
\frac{d}{dt} (\| u\|_{H^m}^2+\| B \|_{H^m}^2)+ \| \nabla u \|_{H^m}^2+\| \nabla B \|_{H^m}^2
\]
\bq\label{apriori2}
\leq C_2 (1+ \| u\|_{L^p}^{\frac{2p}{p-3}} +\| \nabla B \|_{L^{\beta}}^{\frac{2\beta}{2\beta-3}}+\|  u \|_{L^{\infty}}^2+\|  B \|_{L^{\infty}}^2)(\|  u \|_{H^m}^2+\| B \|_{H^m}^2).
\eq
\end{proposition}
\begin{proof}
First, we derive \eqref{apriori1}. Take operator $\nabla$ on equations \eqref{eq1-1} and \eqref{eq3-1}, respectively, take scalar product of them with $\nabla u$ and $\nabla B$, respectively and add them together, we obtain
\[
\frac12 \frac{d}{dt} (\| \nabla u \|_{L^2}^2 +\| \nabla B \|_{L^2}^2)+ \| \Delta u \|_{L^2}^2 +\| \Delta B \|_{L^2}^2
\]
\[
= -\int_{\R^3} \nabla ((\nabla \times B)\times B)\cdot \nabla (\nabla \times B) dx -\int_{\R^3} \nabla (u\cdot \nabla B) \cdot \nabla B dx
\]
\[
-\int_{\R^3} \nabla (u\cdot \nabla u)\cdot \nabla u dx+\int_{\R^3} \nabla(B \cdot \nabla u)\cdot \nabla B dx+\int_{\R^3} \nabla (B\cdot \nabla B) \cdot \nabla u dx
\]
\bq\label{i-one-five}
:=I_1+I_2+I_3+I_4+I_5.
\eq
We use the following cancellation for the estimate of $I_1$ as in \cite{Chae-Degond-Liu} and in Proposition \ref{prp1}
\[I_1=-\sum_{i=1}^3\int_{\R^3} [ \partial_i((\nabla \times B)\times B)-(\partial_i (\nabla \times B))\times B] \cdot \partial_i (\nabla \times B) dx\]
\[=-\sum_{i=1}^3 \int_{\R^3} (\nabla \times B)\times \partial_i B \cdot \partial_i (\nabla \times B) dx.\]

Using Young's inequality and the interpolation inequality, we obtain
\begin{eqnarray*}
|I_1| &\leq & C \| \nabla B \|_{L^{\beta}} \| \nabla B \|_{L^{\frac{2\beta}{\beta-2}}} \| \Delta B \|_{L^2}\\
&\leq & C\| \nabla B \|_{L^{\beta}} \| \nabla B \|_{L^2}^{\frac{\beta-3}{\beta}} \| \Delta B \|_{L^2}^{\frac{\beta+3}{\beta}}\\
&\leq & C\| \nabla B \|_{L^{\beta}}^{\frac{2\beta}{\beta-3}}\| \nabla B \|_{L^2}^2+\frac18 \| \Delta B \|_{L^2}^2.
\end{eqnarray*}
By integration by parts, we can rewrite and estimate the second and third terms on the right hand side of \eqref{i-one-five} as follows:

\begin{eqnarray*}
|I_2|& =&\left|\int_{\R^3} (u\cdot \nabla)B \cdot \Delta B dx\right|\\
 &\leq & C\| u\|_{L^p}\| \nabla B \|_{L^{\frac{2p}{p-2}}} \| \Delta B \|_{L^2}\\
&\leq & C\|u \|_{L^{p}} \| \nabla B \|_{L^2}^{\frac{p-3}{p}} \| \Delta B \|_{L^2}^{\frac{p+3}{p}}\\
&\leq & C\| u \|_{L^{p}}^{\frac{2p}{p-3}}\| \nabla B \|_{L^2}^2+\frac18 \| \Delta B \|_{L^2}^2,
\end{eqnarray*}
and
\begin{eqnarray*}
|I_3|&=&\left|\int_{\R^3} (u\cdot \nabla)u\cdot \Delta u dx\right|\\
 &\leq & C\| u\|_{L^p}\| \nabla u \|_{L^{\frac{2p}{p-2}}} \| \Delta u \|_{L^2}\\
&\leq & C\|u \|_{L^{p}} \| \nabla u \|_{L^2}^{\frac{p-3}{p}} \| \Delta u \|_{L^2}^{\frac{p+3}{p}}\\
&\leq & C\| u \|_{L^{p}}^{\frac{2p}{p-3}}\| \nabla u \|_{L^2}^2+\frac18 \| \Delta u \|_{L^2}^2,
\end{eqnarray*}
For the sum of $I_4$ and $I_5$, we use the following cancellation property :
\[
I_4+I_5=\sum_{i=1}^3 \int_{\R^3} (B \cdot \nabla)\partial_i u\cdot \partial_i B + (\partial_i B \cdot \nabla) u\cdot \partial_i B dx+\sum_{i=1}^3\int_{\R^3} (\partial_i B \cdot \nabla)B \cdot \partial_i u+(B \cdot\nabla) \partial_i B \cdot \partial_i u dx
\]
\[
=\sum_{i=1}^3\int_{\R^3}  (\partial_i B \cdot \nabla) u\cdot \partial_i B + (\partial_i B \cdot \nabla)B \cdot \partial_i udx\]
\[=-\sum_{i=1}^3 \int_{\R^3} (\partial_i B \cdot \nabla)\partial_i B \cdot u +(\partial_i^2 B\cdot\nabla)B \cdot u +(\partial_i B \cdot \nabla)\partial_i B \cdot u dx.
\]
Therefore, we estimate
\begin{eqnarray*}
|I_4+I_5|  &\leq & C\| u\|_{L^p}\| \nabla B \|_{L^{\frac{2p}{p-2}}} \| \Delta B \|_{L^2}\\
&\leq & C\|u \|_{L^{p}} \| \nabla B \|_{L^2}^{\frac{p-3}{p}} \| \Delta B \|_{L^2}^{\frac{p+3}{p}}\\
&\leq & C\| u \|_{L^{p}}^{\frac{2p}{p-3}}\| \nabla B \|_{L^2}^2+\frac18 \| \Delta B \|_{L^2}^2.
\end{eqnarray*}
Summing up the above estimates, we easily deduce \eqref{apriori1}.\\
Next, we derive \eqref{apriori2}.
Similarly to Proposition 3.2 in \cite{Chae-Degond-Liu} and \eqref{i-one-five}, we have the following for $m \geq2$
\[
\frac12 \frac{d}{dt} (\|  u \|_{H^m}^2 +\|  B \|_{H^m}^2)+ \| \nabla u \|_{H^m}^2 +\| \nabla B \|_{H^m}^2
\]
\[
= -\sum_{1\leq |\alpha| \leq m}\int_{\R^3} D^{\alpha} ((\nabla \times B)\times B)\cdot D^{\alpha} (\nabla \times B) dx -\sum_{1\leq |\alpha| \leq m}\int_{\R^3} D^{\alpha} (u\cdot \nabla B) \cdot D^{\alpha} B dx
\]
\[
-\sum_{1\leq |\alpha| \leq m}\int_{\R^3} D^{\alpha} (u\cdot \nabla u)\cdot D^{\alpha} u dx+\sum_{1\leq |\alpha| \leq m}\int_{\R^3} D^{\alpha}(B \cdot \nabla u)\cdot D^{\alpha} B dx\]
\[+\sum_{1\leq |\alpha| \leq m}\int_{\R^3}D^{\alpha} (B\cdot \nabla B) \cdot D^{\alpha} u dx
\]
\bq\label{i-one-five-2}
:=J_1+J_2+J_3+J_4+J_5,
\eq
where $\alpha=(\alpha_1, \alpha_2, \alpha_3)\in {\mathbb{N}}^3$ is a multi-index, $D^{\alpha}=\partial^{|\alpha|}/\partial x_1^{\alpha_1}\partial x_{2}^{\alpha_2}\partial x_3^{\alpha_3}$ and $|\alpha|=\alpha_1+\alpha_2+\alpha_3.$
We successively estimate $J_1,\cdots,J_5$. By the cancellation property such that
\[
J_1=-\sum_{2\leq |\alpha| \leq m}\int_{\R^3} \{D^{\alpha} [(\nabla \times B)\times B]-[D^{\alpha}(\nabla \times B)]\times B\}\cdot D^{\alpha} (\nabla \times B) dx
\]

From the calculus inequality, the interpolation inequality and Young's inequality, we have
\begin{eqnarray*}
|J_1| &\leq & C \| \nabla B \|_{L^{\beta}} \|  B \|_{W^{m, \frac{2\beta}{\beta-2}}} \| \nabla B \|_{H^m}\\
&\leq & C\| \nabla B \|_{L^{\beta}} \|  B \|_{H^m}^{\frac{\beta-3}{\beta}} \|  \nabla B \|_{H^m}^{\frac{\beta+3}{\beta}}\\
&\leq & C\| \nabla B \|_{L^{\beta}}^{\frac{2\beta}{\beta-3}}\|  B \|_{H^m}^2+\frac18 \|  \nabla B \|_{H^m}^2.
\end{eqnarray*}
Similarly, the other terms can be estimated from the Leibniz formula and Young's inequality :
\begin{eqnarray*}
|J_2| &\leq & C(\| u\|_{L^{\infty}}\| B\|_{H^m}+ \| B\|_{L^{\infty}} \| u \|_{H^m})\| \nabla B \|_{H^m}\\
&\leq & C( \| u\|_{L^{\infty}}^2+ \| B\|_{L^{\infty}}^2)(\| u \|_{H^m}^2+\| B\|_{H^m}^2) +\frac18 \| \nabla B \|_{H^m}^2,
\end{eqnarray*}
\[
|J_3| \leq  C( \| u\|_{L^{\infty}}^2+ \|u\|_{L^{\infty}}^2)(\| u \|_{H^m}^2+\|u\|_{H^m}^2) +\frac18 \| \nabla u \|_{H^m}^2,
\]
\[
|J_4| \leq  C( \| u\|_{L^{\infty}}^2+ \|B\|_{L^{\infty}}^2)(\| u \|_{H^m}^2+\|B\|_{H^m}^2) +\frac18 \| \nabla B \|_{H^m}^2,
\]
and
\[
|J_5| \leq  C \| B\|_{L^{\infty}}^2\| B \|_{H^m}^2 +\frac18 \| \nabla u \|_{H^m}^2,
\]
Collecting all the estimates together, we deduce \eqref{apriori2}. This completes the proof.
\end{proof}
We can now prove Theorem \ref{thm1} and Theorem \ref{thm2}.
\begin{pfthm1}
Using Gronwall's Lemma to \eqref{apriori1}, we have the following inequality for all $T\leq T^{*}$
\[
\sup_{ 0<t <T} (\| u(t)\|_{H^1}^2+\| B(t)\|_{H^1}^2) +\int_0^T \| \nabla u\|_{H^1}^2 +\| \nabla B \|_{H^1}^2 dt
\]
\bq\label{gron-1}
\leq (\| u_0 \|_{H^1}^2 +\| B_0\|_{H^1}^2) (1+C_1 \int_0^T \| u\|_{L^p}^{\frac{2p}{p-3}} + \| \nabla B \|_{L^{\beta}}^{\frac{2\beta}{2\beta-3}} )\exp \left( C_1 \int_0^T  \| u\|_{L^p}^{\frac{2p}{p-3}} + \| \nabla B \|_{L^{\beta}}^{\frac{2\beta}{2\beta-3}} \right).
\eq
Again, using Gronwall's Lemma to \eqref{apriori2}, we obtain
\[
\sup_{ 0<t <T} (\| u(t)\|_{H^m}^2+\| B(t)\|_{H^m}^2) +\int_0^T \| \nabla u\|_{H^m}^2 +\| \nabla B \|_{H^m}^2 dt
\]
\[
\leq (\| u_0 \|_{H^m}^2 +\| B_0\|_{H^m}^2) (1+C_2 \int_0^T \| u\|_{L^p}^{\frac{2p}{p-3}} + \| \nabla B \|_{L^{\beta}}^{\frac{2\beta}{2\beta-3}}+\| u\|_{L^{\infty}}^2+\| B \|_{L^{\infty}}^2 dt)\]
\bq\label{gron-2}\times \exp \left( C_2 \int_0^T  \| u\|_{L^p}^{\frac{2p}{p-3}} + \| \nabla B \|_{L^{\beta}}^{\frac{2\beta}{2\beta-3}}+\| u\|_{L^{\infty}}^2+\| B \|_{L^{\infty}}^2 dt\right).
\eq
The interpolation inequality  $\| f\|_{L^{\infty}} \leq C \| f\|_{L^2}^{\frac14} \| \nabla f \|_{H^1}^{\frac34}$ produces us with
\begin{equation}\label{2-7}
\int_0^T \| u \|_{L^{\infty}}^2 dt \leq C \sup_{0\leq t\leq T} \| u(t) \|_{L^2}^{\frac12} \int_0^T \| \nabla u \|_{H^1}^{\frac32} dt \leq C \| u_0 \|_{L^2}^{\frac12} \left( \int_0^T \| \nabla u \|_{H^1}^2 dt  \right)^{\frac34}T^{\frac14},
\end{equation}
and
\begin{equation}\label{2-8}
\int_0^T \| B \|_{L^{\infty}}^2 dt \leq C \sup_{0\leq t\leq T} \|B (t)\|_{L^2}^{\frac12} \int_0^T \| \nabla B\|_{H^1}^{\frac32} dt\leq  C \| B_0 \|_{L^2}^{\frac12} \left( \int_0^T \| \nabla B\|_{H^1}^2 dt  \right)^{\frac34}T^{\frac14}.
\end{equation}
From \eqref{gron-1}--\eqref{2-8}, we obtain that if
\[
\int_0^{T^{*}}  \| u\|_{L^p}^{\frac{2p}{p-3}} + \| \nabla B \|_{L^{\beta}}^{\frac{2\beta}{2\beta-3}} dt< \infty,\mbox{ and } u_0,\,\, B_0 \in H^m(\R^3),
\]
then
\[
\sup_{ 0<t <T^{*}} (\| u(t)\|_{H^m}^2+\| B(t)\|_{H^m}^2) +\int_0^T \| \nabla u\|_{H^m}^2 +\| \nabla B \|_{H^m}^2 dt< \infty.
\]
Since $\frac{2p}{2p-3} \leq q$ and $\frac{2\beta}{2\beta-3} \leq \gamma$, it completes the proof of Theorem \ref{thm1}.
\end{pfthm1}
\begin{pfthm2}
First, we recall the logarithmic Sobolev inequality using BMO space (see \cite[Corollary 2.4]{ogawa}).
\bq\label{ogawa-ineq}
\| \nabla B \|_{L^{\infty}} \leq C(q) \left(1+\| \nabla B \|_{BMO}\left(\ln^{+}( \| \nabla B \|_{W^{1,q}} +\| B \|_{L^{\infty}})\right)^{\frac12}   \right),
\eq
if $\nabla f \in W^{1,q}(\R^3)\cap L^2(\R^3)$ for $3<q$. Since $H^{m-1} (\R^3) \hookrightarrow W^{1,q}(\R^3)$ for some $q>3$ when $m\geq 3$ and $H^m (\R^3) \hookrightarrow L^{\infty}(\R^3)$, we have
\bq\label{ogawa-ineq-2}
\| \nabla B \|_{L^{\infty}} \leq C \left(1+\| \nabla B \|_{BMO}  \ln^{\frac12}(e+ \| B \|_{H^m})\right).
\eq
We estimate each term $J_1,\cdots, J_5$ in \eqref{i-one-five-2} for an integer $m \geq 3$.
Using cancellation property, Young's inequality and \eqref{ogawa-ineq-2}, we have
\begin{eqnarray}
|J_1| &\leq & C \| \nabla B \|_{L^{\infty}} \| B \|_{H^m} \| \nabla B \|_{H^m} \leq C \| \nabla B \|_{L^{\infty}}^2 \| B \|_{H^m}^2 +\frac18 \| \nabla B \|_{H^m}^2\nonumber\\
&\leq & C\left(1+ \| \nabla B \|_{BMO}^2 \ln (e+ \| B \|_{H^m}) \right) \| B \|_{H^m}^2 +\frac18 \| \nabla B \|_{H^m}^2.\label{bmo-1}
\end{eqnarray}
We recall the bilinear estimates in BMO space (see \cite[Lemma1]{KT}).
\bq\label{kt-1}
\| \partial^{\alpha} f \cdot \partial^{\beta} g \|_{L^2} \leq C( \| f\|_{BMO} \| (-\Delta)^{\frac{|\alpha|+|\beta|}{2}} g \|_{L^2}+ \| (-\Delta)^{\frac{|\alpha|+|\beta|}{2}} f \|_{L^2} \| g\|_{BMO}),
\eq
for all $f,\, g \in BMO\cap H^{|\alpha|+|\beta|}$, when $\alpha=(\alpha_1,\, \alpha_2,\, \alpha_3)$ and $\beta=(\beta_1,\, \beta_2,\, \beta_3)$ are multi-indices with $|\alpha|,\, |\beta| \geq 1$.
Using the cancellation property \[\int_{\R^3}  (u\cdot \nabla) D^{\alpha}B \cdot D^{\alpha} B dx = \int_{\R^3}  (u\cdot \nabla) D^{\alpha}u \cdot D^{\alpha} u dx=0,\] and \eqref{kt-1},
we have
\begin{eqnarray}
|J_2| &=& \left|\sum_{1\leq |\alpha| \leq m}\int_{\R^3}[ D^{\alpha} (u\cdot \nabla B) \cdot D^{\alpha} B-(u\cdot \nabla) D^{\alpha}B \cdot D^{\alpha} B] dx    \right|\nonumber\\
&\leq& C( \| u \|_{BMO} \| \nabla B \|_{H^m} + \| B \|_{BMO} \| \nabla u \|_{H^m}) \| B \|_{H^m}\nonumber\\
&\leq & C(\| u \|_{BMO}^2+ \| B \|_{BMO}^2)\| B \|_{H^m}^2 +\frac18 (\| \nabla u \|_{H^m}^2+\| \nabla B \|_{H^m}^2),\label{bmo-2}
\end{eqnarray}
and
\begin{eqnarray}
|J_2| &=& \left|\sum_{1\leq |\alpha| \leq m}\int_{\R^3}[ D^{\alpha} (u\cdot \nabla u) \cdot D^{\alpha} u-(u\cdot \nabla) D^{\alpha}u \cdot D^{\alpha} u] dx    \right|\nonumber\\
&\leq& C \| u \|_{BMO} \| \nabla u \|_{H^m}  \| u \|_{H^m}\nonumber\\
&\leq & C\| u \|_{BMO}^2\| u \|_{H^m}^2 +\frac18 \| \nabla u \|_{H^m}^2.\label{bmo-3}
\end{eqnarray}
Using the cancellation property
\[
\int_{\R^3} (B \cdot \nabla) D^{\alpha}u\cdot D^{\alpha} B dx+\int_{\R^3} (B\cdot \nabla) D^{\alpha}B \cdot D^{\alpha} u dx=0,
\]
we rewrite
\[
J_4+J_5 = \sum_{1\leq |\alpha| \leq m}\int_{\R^3} [D^{\alpha}(B \cdot \nabla u)\cdot D^{\alpha} B-(B \cdot \nabla) D^{\alpha}u\cdot D^{\alpha} B] dx\]
\[+\sum_{1\leq |\alpha| \leq m}\int_{\R^3}[D^{\alpha} (B\cdot \nabla B) \cdot D^{\alpha} u-(B\cdot \nabla) D^{\alpha}B \cdot D^{\alpha} u] dx.
\]
Then, using \eqref{kt-1}, we deduce that
\[
|J_4+J_5| \leq  C(\| B \|_{BMO} \| \nabla u \|_{H^m} + \| u \|_{BMO} \| \nabla B \|_{H^m}) \| B \|_{H^m}+C\| B \|_{BMO} \| \nabla B \|_{H^m}\| u \|_{H^m}\]
\bq\leq  C(\| u \|_{BMO}^2+\| B \|_{BMO}^2)(\| u \|_{H^m}^2+\| B \|_{H^m}^2) +\frac18( \| \nabla u \|_{H^m}^2+ \| \nabla B \|_{H^m}^2).\label{bmo-4}
\eq
Using \eqref{bmo-1}, \eqref{bmo-2}, \eqref{bmo-3} and \eqref{bmo-4},
inequality \eqref{i-one-five-2} can be rewritten as
\[
\frac{d}{dt} (\| u\|_{H^m}^2+\| B \|_{H^m}^2)+ \| \nabla u \|_{H^m}^2+\| \nabla B \|_{H^m}^2
\]
\bq\label{apriori3}
\leq C (1+ \| u\|_{BMO}^{2}+\| B \|_{BMO}^2 +\| \nabla B \|_{BMO}^2\ln (e+ \| B \|_{H^m}))(\|  u \|_{H^m}^2+\| B \|_{H^m}^2).
\eq

 Let $X(t)= e+\| u\|_{H^m}^2+\| B\|_{H^m}^2$. Then we rewrite \eqref{apriori3} into
 \[
 \frac{d}{dt}X(t) \leq C(1+ \| u\|_{BMO}^2+\| \nabla B\|_{BMO}^2)X(t) \ln (e+X(t)).
 \]
 Using Gronwall type Lemma, we obtain
 \[
 \sup_{0\leq t\leq T}  X(t) \leq (e+ \| u_0 \|_{H^m}^2 +\| B_0 \|_{H^m}^2) \exp \left( C\exp \left(\int_0^{T} \| u\|_{BMO}^2+\| \nabla B\|_{BMO}^2dt    \right)\right).
 \]
 This completes the proof of Theorem \ref{thm2}.
\end{pfthm2}
\section{Small data global existence}
\begin{pfthm3}
Denote $\Lambda=(-\Delta )^{\frac12}$. If we take operator $\Lambda^{\frac12}$ on the both sides of \eqref{eq1-1} and \eqref{eq3-1}, take scalar product with $\Lambda^{\frac12} u$ and $\Lambda^{\frac12} B$, respectively, and add these to obtain that
\begin{eqnarray}
 &&\frac12 \frac{d}{dt} (\| \Lambda^{\frac12} u \|_{L^2}^2+\| \Lambda^{\frac12} B \|_{L^2}^2) +\| \Lambda^{\frac32} u \|_{L^2}^2 +\| \Lambda^{\frac32} B \|_{L^2}^2 \nonumber\\
 && \leq C (\|  (u \cdot \nabla) u \|_{L^{\frac32}}\| \Lambda u\|_{L^3}+\| (B \cdot \nabla)B \|_{L^{\frac32}}\| \Lambda u \|_{L^3}\nonumber\\
 && +\|  (u \cdot \nabla) B\|_{L^{\frac32}}\| \Lambda B\|_{L^3}+\|  (B \cdot \nabla) u \|_{L^{\frac32}}\| \Lambda B\|_{L^3})\nonumber\\
 &&+ C\| \nabla \times B \|_{L^6} \| B \|_{L^3} \| \Lambda \nabla \times B \|_{L^2}\nonumber\\
 &&\leq C( \| \Lambda^{\frac12} u \|_{L^2}+ \| \Lambda^{\frac12} B \|_{L^2})\left(\| \Lambda^{\frac32} u \|_{L^2}^2 +\| \Lambda^{\frac32} B \|_{L^2}^2  \right)\nonumber\\
 &&+C \| \Lambda^{\frac12} B \|_{L^2} \| \Lambda^{\frac52} B \|_{L^2}^2. \label{ineq-00}
\end{eqnarray}
    If we take operator $\Lambda^{\frac32}$ on the both sides of \eqref{eq1-1} and \eqref{eq3-1} and take scalar product with $\Lambda^{\frac32} u$ and $\Lambda^{\frac32} B$, respectively, we deduce that
\begin{eqnarray}
&&\frac12 \frac{d}{dt} \| \Lambda^{\frac32} u \|_{L^2}^2 + \| \Lambda^{\frac52} u \|_{L^2}^2 \nonumber\\
&& \leq -\int_{\R^3} \{ \Lambda^{\frac32} [(u\cdot \nabla)u]-(u\cdot\nabla) \Lambda^{\frac32} u\} \cdot \Lambda^{\frac32} u dx\nonumber\\
&& +\int_{\R^3} \{ \Lambda^{\frac32} [(B\cdot \nabla)B]-(B\cdot\nabla) \Lambda^{\frac32} B\} \cdot \Lambda^{\frac32} u dx+\int_{\R^3}[(B \cdot \nabla)\Lambda^{\frac32} B ]\cdot \Lambda^{\frac32} u dx,\label{ineq-10}
\end{eqnarray}
and
\begin{eqnarray}
&&\frac12 \frac{d}{dt} \| \Lambda^{\frac32} B \|_{L^2}^2 + \| \Lambda^{\frac52} B \|_{L^2}^2 \nonumber\\
&& \leq -\int_{\R^3} \{ \Lambda^{\frac32} [(u\cdot \nabla)B]-(u\cdot\nabla) \Lambda^{\frac32} B\} \cdot \Lambda^{\frac32} B dx\nonumber\\
&& +\int_{\R^3} \{ \Lambda^{\frac32} [(B\cdot \nabla)u]-(B\cdot\nabla) \Lambda^{\frac32} u\} \cdot \Lambda^{\frac32} B dx+\int_{\R^3}[(B \cdot \nabla)\Lambda^{\frac32} u ]\cdot \Lambda^{\frac32} B dx\nonumber\\
&&-\int_{\R^3} \{ \Lambda^{\frac32} [(\nabla\times B)\times B]-(\Lambda^{\frac32} \nabla\times B)\times B\} \cdot \Lambda^{\frac32}\nabla \times  B dx.\label{ineq-20}
\end{eqnarray}
Adding \eqref{ineq-10} and \eqref{ineq-20}, and using the fact that
\[
\int_{\R^3}[(B \cdot \nabla)\Lambda^{\frac32} B ]\cdot \Lambda^{\frac32} u dx+\int_{\R^3}[(B \cdot \nabla)\Lambda^{\frac32} u ]\cdot \Lambda^{\frac32} B dx=0,
\]
We obtain
\begin{eqnarray}
&&\frac12 \frac{d}{dt} (\| \Lambda^{\frac32} u \|_{L^2}^2 +\| \Lambda^{\frac32} B \|_{L^2}^2)+\|\Lambda^{\frac52} u \|_{L^2}^2+\| \Lambda^{\frac52} B \|_{L^2}^2\nonumber\\
&& \leq  \| \Lambda^{\frac32} [(u \cdot \nabla) u]-(u \cdot \nabla) \Lambda^{\frac32} u \|_{L^2} \| \Lambda^{\frac32} u \|_{L^2}\nonumber\\
&&+\| \Lambda^{\frac32} [(u \cdot \nabla) B]-(u \cdot \nabla) \Lambda^{\frac32} B \|_{L^2} \| \Lambda^{\frac32} B \|_{L^2}\nonumber\\
&&+\| \Lambda^{\frac32} [(B \cdot \nabla) B]-(B \cdot \nabla) \Lambda^{\frac32} B \|_{L^2} \| \Lambda^{\frac32} u \|_{L^2}\nonumber\\
&&+\| \Lambda^{\frac32} [(B \cdot \nabla) u]-(B \cdot \nabla) \Lambda^{\frac32} u \|_{L^2} \| \Lambda^{\frac32} B \|_{L^2}\nonumber\\
&&+\| \Lambda^{\frac32} [( \nabla \times B)\times B]-(\Lambda^{\frac32} \nabla\times B)\times B \|_{L^2} \| \Lambda^{\frac32}\nabla \times  B \|_{L^2}.\label{ineq-30}
\end{eqnarray}
We recall the commutator estimate(\cite{kpv})
\[
\| \Lambda^{s}(fg) -f \Lambda^s g \|_{L^p} \leq C( \| \nabla f \|_{L^{q_1}} \| \Lambda^{s-1} g \|_{L^{r_1}} +\| \Lambda^{s} f \|_{L^{q_2}} \| g\|_{L^{r_2}},
\]
where $\frac{1}{p}=\frac{1}{q_i}+\frac{1}{r_i}$, $i=1,2$ and $ p,\, q_i ,\, r_i \in [1, \infty]$.\\
Using above commutator estimate together with the Sobolev inequalities, we deduce
\[
\| \Lambda^{\frac32} [(u \cdot \nabla) u]-(u \cdot \nabla) \Lambda^{\frac32} u \|_{L^2} \leq C (\| \Lambda^{\frac32} u \|_{L^6} \| \nabla u\|_{L^3} + \| \Lambda^{\frac12} u \|_{L^6} \| \Lambda^2 u \|_{L^6})\leq C\| \Lambda^{\frac52} u \|_{L^2} \| \Lambda^{\frac32} u \|_{L^2},
\]
\[
\| \Lambda^{\frac32} [(u \cdot \nabla) B]-(u \cdot \nabla) \Lambda^{\frac32} B \|_{L^2} \leq C(\| \Lambda^{\frac52} u \|_{L^2} \| \Lambda^{\frac32} B \|_{L^2}+\| \Lambda^{\frac52} B \|_{L^2} \| \Lambda^{\frac32} u \|_{L^2}),
\]
\[
\| \Lambda^{\frac32} [(B \cdot \nabla) B]-(B \cdot \nabla) \Lambda^{\frac32} B \|_{L^2} \leq C\| \Lambda^{\frac52} B \|_{L^2} \| \Lambda^{\frac32} B \|_{L^2},
\]
\[
\| \Lambda^{\frac32} [(B \cdot \nabla) u]-(B \cdot \nabla) \Lambda^{\frac32} u \|_{L^2} \leq C(\| \Lambda^{\frac52} u \|_{L^2} \| \Lambda^{\frac32} B \|_{L^2}+\| \Lambda^{\frac52} B \|_{L^2} \| \Lambda^{\frac32} u \|_{L^2}),
\]
and
\[
\| \Lambda^{\frac32} [(\nabla \times B)\times B]-(\Lambda^{\frac32}\nabla\times B) \times B \|_{L^2} \leq C\| \Lambda^{\frac52} B \|_{L^2} \| \Lambda^{\frac32} B \|_{L^2}.
\]
Hence we obtain that
\begin{eqnarray}
&&\frac12 \frac{d}{dt} (\| \Lambda^{\frac32} u \|_{L^2}^2 +\| \Lambda^{\frac32} B \|_{L^2}^2)+\|\Lambda^{\frac52} u \|_{L^2}^2+\| \Lambda^{\frac52} B \|_{L^2}^2\nonumber\\
&& \leq C(\| \Lambda^{\frac32} u \|_{L^2} +\| \Lambda^{\frac32} B \|_{L^2})\left(\|\Lambda^{\frac52} u \|_{L^2}^2+\| \Lambda^{\frac52} B \|_{L^2}^2 +\| \Lambda^{\frac32} u \|_{L^2}^2 +\| \Lambda^{\frac32} B \|_{L^2}^2 \right).\label{ineq-40}
\end{eqnarray}
Adding \eqref{ineq-00} and \eqref{ineq-40},  we have
\[
\frac12 \frac{d}{dt} (\| \Lambda^{\frac12} u(t) \|_{L^2}^2+\| \Lambda^{\frac12} B(t) \|_{L^2}^2+\| \Lambda^{\frac32} u(t) \|_{L^2}^2+\| \Lambda^{\frac32} B(t) \|_{L^2}^2)
\]
\[
+\left(1-C\left(\| \Lambda^{\frac12} u(t) \|_{L^2}+\| \Lambda^{\frac12} B(t) \|_{L^2}+\| \Lambda^{\frac32} u(t) \|_{L^2}+\| \Lambda^{\frac32} B(t) \|_{L^2}\right)\right)
\]
\[
\times \left( \| \Lambda^{\frac32} u(t) \|_{L^2}^2+\| \Lambda^{\frac32} B(t) \|_{L^2}^2+\| \Lambda^{\frac52} u(t) \|_{L^2}^2+\| \Lambda^{\frac52} B(t) \|_{L^2}^2                \right)\leq 0.
\]
Choosing $K$ so small that (by interpolation of $H^{\frac12}$ between $\dot{H}^{\frac32}$ and $L^2$)
\[
C\left( \| \Lambda^{\frac12}u_0 \|_{L^2}+\| \Lambda^{\frac12}B_0 \|_{L^2}+\| \Lambda^{\frac32}u_0 \|_{L^2}+\| \Lambda^{\frac32}B_0 \|_{L^2}\right)\leq \frac12,
\]
then we have for any $T\in (0, T^{*})$ ($T^{*}$ is the maximal time of existence of $H^m$ solution),
\[
(u,\, B) \in L^{\infty}(0, \, T; H^{\frac32}) \cap L^2( 0,\, T ; H^{\frac52}).
\]
Using the fact that
\[
\dot{H}^{\frac32}(\R^3)\hookrightarrow BMO(\R^3),
\]
and  the above estimates, we have
\[
\nabla u,\,\, \nabla B \in  L^2(0, T; , BMO),\mbox{ for all }T \in (0, T^{*}),
\]
which satisfies the integrability condition in Theorem \ref{thm2}. Consequently, using the continuation argument, it completes the proof of Theorem \ref{thm3}.
\end{pfthm3}
\begin{pfthm4}
For the simplicity, we set
\[
\pi=p+\frac{|B|^2}{2}.
\]
Applying operator $\Delta_q$ to \eqref{eq1-1} and \eqref{eq3-1}, respectively, we infer that
\begin{eqnarray}
&& \partial_t \Delta_q u +(u \cdot \nabla)\Delta_q u-\Delta \Delta_q u+\nabla \Delta_q \pi\nonumber\\
&&=-[\Delta_q,\, u]\cdot \nabla u+\Delta_q (B\cdot \nabla B),\label{deltaq1}
\end{eqnarray}
and
\begin{eqnarray}
&& \partial_t \Delta_q B +(u \cdot \nabla)\Delta_q B-\Delta \Delta_q B+\nabla \times\Delta_q (j \times B)\nonumber\\
&&=-[\Delta_q,\, u]\cdot \nabla B+\Delta_q(B \cdot \nabla u),\label{deltaq2}
\end{eqnarray}
Let $T^{*}$ be the maximal time of existence of solution such that $(u(t), B(t))\in H^m(\R^3)$, $m >\frac52$ for all $t \in [0, T^{*})$.
Multiplying $\Delta_q u$ and $\Delta_q B$ to the both sides of \eqref{deltaq1} and \eqref{deltaq2}, respectively, and integrating over $\R^3$, we have for $t \in (0,\, T^{*})$
\begin{eqnarray}
&&\frac12 \frac{d}{dt} \| \Delta_q u \|_{L^2}^2 +C 2^{2q} \| \Delta_q u \|_{L^2}^2\nonumber\\
&&\leq \| [\Delta_q, u]\cdot \nabla u\|_{L^2} \| \Delta_q u \|_{L^2}+\| \Delta_q( B\cdot \nabla B) \|_{L^2}\| \Delta_q u\|_{L^2},\label{ener-1}
\end{eqnarray}
and
\begin{eqnarray}
&&\frac12  \frac{d}{dt} \| \Delta_q B \|_{L^2}^2 +C 2^{2q} \| \Delta_q B\|_{L^2}^2\leq \| [\Delta_q, u]\cdot \nabla B\|_{L^2} \| \Delta_q B \|_{L^2}\nonumber\\&&+C2^q\| \Delta_q( j\times B) \|_{L^2}\| \Delta_q B\|_{L^2} +\| \Delta_q (B \cdot \nabla u) \|_{L^2} \|\Delta_q B\|_{L^2},\label{ener-2}
\end{eqnarray}
Dividing both sides of \eqref{ener-1} and \eqref{ener-2} by $\|\Delta_q u\|_{L^2}$ and $\| \Delta_q B \|_{L^2}$, respectively and adding these, we obtain that
\begin{eqnarray}
&&\frac{d}{dt} ( \| \Delta_q u \|_{L^2} +\| \Delta_q B \|_{L^2}) +C 2^{2q} (\| \Delta_q u \|_{L^2}+\| \Delta_q B \|_{L^2})\nonumber\\
&&\leq \| [\Delta_q, u]\cdot \nabla u\|_{L^2} +\| \Delta_q( B\cdot \nabla B) \|_{L^2} + \| [\Delta_q, u]\cdot \nabla B\|_{L^2}\nonumber\\&&+C2^q\| \Delta_q( j\times B) \|_{L^2}+\| \Delta_q (B \cdot \nabla u) \|_{L^2}.\label{ener-3}
\end{eqnarray}

Multiplying $2^{\frac{q}{2}}$ and integrating over $[0, t]$ with $t\leq T^{*}-\delta$ for any $\delta>0$, we have
\begin{eqnarray}
&& 2^{\frac{q}{2}}(\| \Delta_q u(t)\|_{L^2}+\|\Delta_q B(t)\|_{L^2})+C_1\int_0^t2^{\frac{5q}{2}}( \|\Delta_q u(s)\|_{L^2}+\| \Delta_q B(s)\|_{L^2})ds\nonumber\\
&& \leq \int_0^t2^{\frac{q}{2}}\| [\Delta_q, u]\cdot \nabla u\|_{L^2}ds +\int_0^t2^{\frac{q}{2}}\| \Delta_q( B\cdot \nabla B) \|_{L^2}ds + \int_0^t2^{\frac{q}{2}}\| [\Delta_q, u]\cdot \nabla B\|_{L^2}ds\nonumber\\&&+C\int_0^t2^{\frac{3q}{2}}\| \Delta_q( j\times B) \|_{L^2}ds+\int_0^t2^{\frac{q}{2}}\| \Delta_q (B \cdot \nabla u) \|_{L^2}ds+2^{\frac{q}{2}}(\| \Delta_q u_0\|_{L^2}+\|\Delta_q B_0\|_{L^2}).\nonumber
\end{eqnarray}

Taking summation over $q \in {\mathbb{Z}}$, we have
\begin{eqnarray}
&& (\| u(t)\|_{\dot{B}^{\frac12}_{2,1}}+\| B(t)\|_{\dot{B}^{\frac12}_{2,1}})+C_1\int_0^t( \| u(s)\|_{\dot{B}^{\frac52}_{2,1}}+\| B(s)\|_{\dot{B}^{\frac52}_{2,1}})ds\nonumber\\
&& \leq \int_0^t\| B \cdot \nabla B \|_{\dot{B}^{\frac12}_{2,1}}ds+\int_0^t\| B \cdot \nabla u \|_{\dot{B}^{\frac12}_{2,1}}ds+ C\int_0^t\| j \times B \|_{\dot{B}^{\frac32}_{2,1}}ds\nonumber\\
&& + \int_0^t\sum_{q} 2^{\frac{q}{2}} \| [\Delta_q, u]\cdot \nabla u\|_{L^2}ds+\int_0^t\sum_{q} 2^{\frac{q}{2}} \| [\Delta_q, u]\cdot \nabla B\|_{L^2}ds\nonumber\\
&&+(\| u_0\|_{\dot{B}^{\frac12}_{2,1}}+\| B_0\|_{\dot{B}^{\frac12}_{2,1}}).\label{ener-4}
\end{eqnarray}
Using (iii) of Proposition \ref{Besov}, we have
\[
\| B \cdot \nabla B \|_{\dot{B}^{\frac12}_{2,1}}\leq C \| B\|_{\dot{B}^{\frac12}_{2,1}}\| \nabla B \|_{\dot{B}^{\frac32}_{2,1}},
\]
\[
\| B \cdot \nabla u \|_{\dot{B}^{\frac12}_{2,1}}\leq C \| B\|_{\dot{B}^{\frac12}_{2,1}}\| \nabla u \|_{\dot{B}^{\frac32}_{2,1}},
\]
and
\[
\| j \times B \|_{\dot{B}^{\frac32}_{2,1}}\leq C \| j\|_{\dot{B}^{\frac32}_{2,1}}\|  B \|_{\dot{B}^{\frac32}_{2,1}}.
\]

From Proposition \ref{Besov} (iv), we obtain
\[
\| [\Delta_q, u]\cdot \nabla u\|_{L^2} \leq c_q 2^{-\frac{q}{2}} \| u \|_{\dot{B}^{\frac52}_{2,1}}\|\nabla u \|_{\dot{B}^{-\frac12}_{2,1}},
\]
and
\[
\| [\Delta_q, u]\cdot \nabla B\|_{L^2} \leq c_q 2^{-\frac{q}{2}} \| u \|_{\dot{B}^{\frac52}_{2,1}}\|\nabla B \|_{\dot{B}^{-\frac12}_{2,1}}.
\]
Then the right hand side of \eqref{ener-4} are estimated as follows :
\begin{eqnarray}
&& (\| u(t)\|_{\dot{B}^{\frac12}_{2,1}}+\| B(t)\|_{\dot{B}^{\frac12}_{2,1}})+C_1\int_0^t( \| u(s)\|_{\dot{B}^{\frac52}_{2,1}}+\| B(s)\|_{\dot{B}^{\frac52}_{2,1}})ds\nonumber\\
&& \leq  C \int_0^t(\| u(s)\|_{\dot{B}^{\frac12}_{2,1}}+\| B(s)\|_{\dot{B}^{\frac12}_{2,1}})(\| u (s)\|_{\dot{B}^{\frac52}_{2,1}}+\|  B(s) \|_{\dot{B}^{\frac52}_{2,1}})ds +C \int_0^t\|  B \|_{\dot{B}^{\frac32}_{2,1}}\| j\|_{\dot{B}^{\frac32}_{2,1}}ds\nonumber\\
&& +(\| u_0\|_{\dot{B}^{\frac12}_{2,1}}+\| B_0\|_{\dot{B}^{\frac12}_{2,1}}).\label{ener-5}
\end{eqnarray}
To estimate the term $\int_0^t\|  B \|_{\dot{B}^{\frac32}_{2,1}}\| j\|_{\dot{B}^{\frac32}_{2,1}}ds$ of \eqref{ener-5}, we rewrite \eqref{eq3-1} as
\[
\partial_t B +(u\cdot\nabla )B+(B\cdot\nabla)j-(j \cdot\nabla)B=(B\cdot \nabla)u+\Delta B.
\]
Applying operator $\Delta_q$ to above, we infer that
\begin{eqnarray}\label{ener-6}
&&\partial_t \Delta_q B +(u \cdot \nabla)\Delta_q B-\Delta \Delta_q B+\Delta_q (B\cdot\nabla j) -(j\cdot\nabla)\Delta_q B\nonumber\\&&=-[\Delta_q,\, u]\cdot \nabla B+\Delta_q(B \cdot \nabla u)+[\Delta_q, j]\cdot \nabla B.
\end{eqnarray}
Multiplying $\Delta_q B$ on the both sides of \eqref{ener-6}, integrating over $\R^3$, we deduce that
\[
\frac12\frac{d}{dt} \| \Delta_q B \|_{L^2}^2 +C 2^{2q} \| \Delta_q B\|_{L^2}^2\leq \| [\Delta_q, u]\cdot \nabla B\|_{L^2} \| \Delta_q B \|_{L^2}
\]
\[
+\| \Delta_q (B \cdot \nabla j)\|_{L^2}\| \Delta_q B \|_{L^2} +\| \Delta_q (B \cdot \nabla u)\|_{L^2}\| \Delta_q B \|_{L^2} +\| [\Delta_q, j]\cdot \nabla B\|_{L^2}\| \Delta_q B \|_{L^2}.
\]
Dividing above by $\| \Delta_q B \|_{L^2}$, multiplying $2^{\frac{3q}{2}}$ and integrating over $[0, t]$ with $t\leq T^{*}-\delta$ for any $\delta>0$, we have
\begin{eqnarray}
&& 2^{\frac{3q}{2}}\|\Delta_q B(t)\|_{L^2}+C_1\int_0^t2^{\frac{7q}{2}}\| \Delta_qB(s)\|_{L^2}ds \leq \int_0^t2^{\frac{3q}{2}}\| \Delta_q (B \cdot \nabla j)\|_{L^2}ds\nonumber\\
&& +\int_0^t2^{\frac{3q}{2}}\| \Delta_q (B \cdot \nabla u)\|_{L^2}ds+ \int_0^t 2^{\frac{3q}{2}} \| [\Delta_q, u]\cdot \nabla B\|_{L^2}ds\nonumber\\&&+\int_0^t 2^{\frac{3q}{2}} \| [\Delta_q, j]\cdot \nabla B\|_{L^2}ds+2^{\frac{3q}{2}}\|\Delta_q B_0\|_{L^2}.\label{ener-7}
\end{eqnarray}

Taking summation over $q \in {\mathbb{Z}}$, we obtain that
\begin{eqnarray}
&& \| B(t)\|_{\dot{B}^{\frac32}_{2,1}}+C_1\int_0^t\| B(s)\|_{\dot{B}^{\frac72}_{2,1}} ds\leq C\int_0^t\| B \cdot \nabla j \|_{\dot{B}^{\frac32}_{2,1}}+C\| B \cdot \nabla u \|_{\dot{B}^{\frac32}_{2,1}}ds\nonumber\\
&& + \int_0^t\sum_{q} 2^{\frac{3q}{2}} \| [\Delta_q, u]\cdot \nabla B\|_{L^2}+\sum_{q} 2^{\frac{3q}{2}} \| [\Delta_q, j]\cdot \nabla B\|_{L^2}ds+\| B_0\|_{\dot{B}^{\frac32}_{2,1}}.\label{ener-7-1}
\end{eqnarray}
Again using Proposition \ref{Besov} (iii), we estimate
\[
\| B \cdot \nabla j \|_{\dot{B}^{\frac32}_{2,1}}\leq C\| B  \|_{\dot{B}^{\frac32}_{2,1}}\|\nabla j \|_{\dot{B}^{\frac32}_{2,1}},
\]
and
\[
\| B \cdot \nabla u \|_{\dot{B}^{\frac32}_{2,1}}\leq C\| B  \|_{\dot{B}^{\frac32}_{2,1}}\|\nabla u \|_{\dot{B}^{\frac32}_{2,1}}.
\]
Using commutator estimates of Proposition \ref{Besov} (iv), we have
\[
\| [\Delta_q, u]\cdot \nabla B\|_{L^2} \leq c_q 2^{-\frac{3q}{2}} \| u \|_{\dot{B}^{\frac52}_{2,1}}\|\nabla B \|_{\dot{B}^{\frac12}_{2,1}},
\]
and
\[
\| [\Delta_q, j]\cdot \nabla B\|_{L^2} \leq c_q 2^{-\frac{3q}{2}} \| j \|_{\dot{B}^{\frac52}_{2,1}}\|\nabla B \|_{\dot{B}^{\frac12}_{2,1}}.
\]
Then, we infer that
\begin{equation}
 \| B(t)\|_{\dot{B}^{\frac32}_{2,1}}+C_1\int_0^t\| B(s)\|_{\dot{B}^{\frac72}_{2,1}}ds \leq C\int_0^t\| B \|_{\dot{B}^{\frac32}_{2,1}}(\| u \|_{\dot{B}^{\frac52}_{2,1}}+\| B\|_{\dot{B}^{\frac72}_{2,1}})ds+\| B_0\|_{\dot{B}^{\frac32}_{2,1}} .\label{ener-8}
\end{equation}
Adding \eqref{ener-5} and \eqref{ener-8}, we obtain
\begin{eqnarray}
&&(\| u(t)\|_{\dot{B}^{\frac12}_{2,1}}+\| B(t)\|_{\dot{B}^{\frac12}_{2,1}}+\| B(t)\|_{\dot{B}^{\frac32}_{2,1}})+C_1 \int_0^t (\| u(t)\|_{\dot{B}^{\frac52}_{2,1}}+\| B(t)\|_{\dot{B}^{\frac52}_{2,1}}+\| B(t)\|_{\dot{B}^{\frac72}_{2,1}}) ds
\nonumber\\&&
 \leq C\int_0^t (\| u(s)\|_{\dot{B}^{\frac12}_{2,1}}+\| B(s)\|_{\dot{B}^{\frac12}_{2,1}}+\| B(s) \|_{\dot{B}^{\frac32}_{2,1}})( \| u(s)\|_{\dot{B}^{\frac52}_{2,1}}+\| B(s)\|_{\dot{B}^{\frac52}_{2,1}}+\| B(s)\|_{\dot{B}^{\frac72}_{2,1}})ds
\nonumber\\&&
 +(\| u_0\|_{\dot{B}^{\frac12}_{2,1}}+\| B_0\|_{\dot{B}^{\frac12}_{2,1}}+\| B_0\|_{\dot{B}^{\frac32}_{2,1}}).\label{ener-las}\end{eqnarray}
 By using the interpolation in Proposition \ref{Besov} (v)
 \[\| B_0\|_{\dot{B}^{\frac12}_{2,1}}\leq C \|B_0\|_{\dot{B}^{\frac32}_{2,1}}^{\frac13} \| B_0\|_{L^2}^{\frac23}\] and choosing $\epsilon$ to be so small, we make
 \bq\label{ini-small}
 C(\| u_0\|_{\dot{B}^{\frac12}_{2,1}}+\| B_0\|_{\dot{B}^{\frac12}_{2,1}}+\| B_0\|_{\dot{B}^{\frac32}_{2,1}})<\frac{C_1}{2}.
 \eq
 Suppose there exists a first time $t \in (0, T^{*})$ such that \[ C(\| u(t)\|_{\dot{B}^{\frac12}_{2,1}}+\| B(t)\|_{\dot{B}^{\frac12}_{2,1}}+\| B(t)\|_{\dot{B}^{\frac32}_{2,1}})\geq \frac{C_1}{2} .\]
 Then,
 \begin{eqnarray}
&& (\| u(t)\|_{\dot{B}^{\frac12}_{2,1}}+\| B(t)\|_{\dot{B}^{\frac12}_{2,1}}+\| B(t)\|_{\dot{B}^{\frac32}_{2,1}})+C_1\int_0^t( \| u(s)\|_{\dot{B}^{\frac52}_{2,1}}+\| B(s)\|_{\dot{B}^{\frac52}_{2,1}}+\| B(s)\|_{\dot{B}^{\frac72}_{2,1}})ds\nonumber\\&&\leq \frac{C_1}{2}\int_0^t( \| u(s)\|_{\dot{B}^{\frac52}_{2,1}}+\| B(s)\|_{\dot{B}^{\frac52}_{2,1}}+\| B(s)\|_{\dot{B}^{\frac72}_{2,1}})ds+(\| (u_0, B_0)\|_{\dot{B}^{\frac12}_{2,1}}+\| B_0\|_{\dot{B}^{\frac32}_{2,1}}).\label{ener-8-1}
\end{eqnarray}
 Absorbing the first term in the right hand sides of \eqref{ener-8-1} into the left hand side, we have
\begin{eqnarray}\label{ener-9}
&& (\| u(t)\|_{\dot{B}^{\frac12}_{2,1}}+\| B(t)\|_{\dot{B}^{\frac12}_{2,1}}+\| B(t)\|_{\dot{B}^{\frac32}_{2,1}})+\frac{C_1}{2}\int_0^t( \| u(s)\|_{\dot{B}^{\frac52}_{2,1}}+\| B(s)\|_{\dot{B}^{\frac52}_{2,1}}+\| B(s)\|_{\dot{B}^{\frac72}_{2,1}})ds\nonumber\\&&\leq (\|( u_0,\, B_0)\|_{\dot{B}^{\frac12}_{2,1}}+\| B_0\|_{\dot{B}^{\frac32}_{2,1}}).
\end{eqnarray}
It contradicts to the definition of $t$ and \eqref{ini-small}.\\
Hence,  we need to have
\[
\sup_{0\leq t \leq T^{*}} (\| u(t)\|_{\dot{B}^{\frac12}_{2,1}}+\| B(t)\|_{\dot{B}^{\frac12}_{2,1}}+\| B(t)\|_{\dot{B}^{\frac32}_{2,1}})\leq \frac{C_1}{2C},
\]
and this, combined with \eqref{ener-las}, implies
\[
\sup_{0\leq t \leq T^{*}} (\| u(t)\|_{\dot{B}^{\frac12}_{2,1}}+\| B(t)\|_{\dot{B}^{\frac12}_{2,1}}+\| B(t)\|_{\dot{B}^{\frac32}_{2,1}}) + \frac{C_1}{2}\int_0^{T^{*}} \| u(t)\|_{\dot{B}^{\frac52}_{2,1}}+\| B(t)\|_{\dot{B}^{\frac52}_{2,1}}+\| B(t)\|_{\dot{B}^{\frac72}_{2,1}}dt
\]
\[
\leq \| u_0\|_{\dot{B}^{\frac12}_{2,1}}+\| B_0\|_{\dot{B}^{\frac12}_{2,1}}+\| B_0\|_{\dot{B}^{\frac32}_{2,1}}.
\]
Using the interpolation and Sobolev embedding, we have
\[
\int_0^{T^{*}} \| u\|_{L^6}^4 dt \leq C\int_0^{T^{*}} \| u \|_{\dot{B}^{1}_{2,1}}^4 dt \leq C\int_0^{T^{*}}\| u\|_{\dot{B}^{\frac12}_{2,1}}^3\| u\|_{\dot{B}^{\frac52}_{2,1}}dt\leq C\sup_{0\leq t\leq T^{*}} \| u\|_{\dot{B}^{\frac12}_{2,1}}^3\int_0^{T^{*}}\| u\|_{\dot{B}^{\frac52}_{2,1}}dt,
\]
and
\[
\int_0^{T^{*}} \| \nabla B\|_{L^6}^4 dt \leq C\int_0^{T^{*}} \| B \|_{\dot{B}^{2}_{2,1}}^4 dt \leq C\int_0^{T^{*}}\| B\|_{\dot{B}^{\frac32}_{2,1}}^3\| u\|_{\dot{B}^{\frac72}_{2,1}}dt\leq C\sup_{0\leq t\leq T^{*}} \| B\|_{\dot{B}^{\frac32}_{2,1}}^3\int_0^{T^{*}}\| B\|_{\dot{B}^{\frac72}_{2,1}}dt.
\]
Hence by the blow-up criterion in Theorem \ref{thm1}, we conclude $T^{*}=\infty$.
\end{pfthm4}
\section*{Acknowledgments}
The research of the first author is supported partially by NRF Grant no. 2006-0093854. The research of the second author is supported partially by NRF Grant no. 2012-0000942. The authors would like to thank Prof. Jishan Fan for useful comment.

\end{document}